\documentclass[12pt,reqno]{amsart}

\usepackage{setspace}
\usepackage[all]{xy}
\usepackage{amssymb,amsmath}

\usepackage{enumitem}
 
\calclayout
\newtheorem{dummy}{anything}[section]
\newtheorem{theorem}[dummy]{Theorem}
\newtheorem{lemma}[dummy]{Lemma}
\newtheorem{proposition}[dummy]{Proposition}

\newtheorem{conjecture}[dummy]{Conjecture}
\theoremstyle{definition}
\newtheorem{definition}[dummy]{Definition}

  \newtheorem{remark}[dummy]{Remark}

\newcommand{\bbZ}{\mathbb Z}
\newcommand{\bbQ}{\mathbb Q}
\newcommand{\bbS}{\mathbb S}
\newcommand{\bbF}{\mathbb F}
\newcommand{\bbC}{\mathbb C}
\newcommand{\cH}{\mathcal H}

\newcommand{\cD}{\mathcal D}
\newcommand{\bV}{\mathbf V}
\newcommand{\bA}{\mathbf A}

\DeclareMathOperator*{\colim}{colim}
\DeclareMathOperator*{\hocolim}{hocolim}

\DeclareMathOperator{\Hom}{Hom}
\DeclareMathOperator{\Ext}{Ext}
\DeclareMathOperator{\Rep}{Rep}
\DeclareMathOperator{\Inn}{Inn}

\DeclareMathOperator{\Map}{Map} 
\DeclareMathOperator{\Or}{Or}
\DeclareMathOperator{\rk}{rk}
\DeclareMathOperator{\Ind}{Ind}
\DeclareMathOperator{\Res}{Res}

\newcommand{\la}{\langle}
\newcommand{\ra}{\rangle} 
\def\maprt#1{\smash{\,\mathop{\longrightarrow}\limits^{#1}\,}}

\begin{document}

\title{Rank Three $p$-Group Actions on Products of Spheres}
\author{Erg\" un Yal\c c\i n}

\address{Department of Mathematics, Bilkent
University, Ankara, 06800, Turkey.}

\email{yalcine@fen.bilkent.edu.tr}



\thanks{Partially supported by T\" UB\. ITAK-B\. IDEB-2219.}

\date{\today}

\begin{abstract} Let $p$ be an odd prime. We prove that every rank three $p$-group acts freely and smoothly on a product of three spheres. To construct this action, we first prove a generalization of a theorem of L\" uck and Oliver on constructions of $G$-equivariant vector bundles. We also give some other applications of this generalization.
\end{abstract}

\maketitle

\section{Introduction} 
One of the classical problems in transformation group theory is the problem of classifying all finite groups that can act freely on a product of $k$ spheres for an arbitrary positive integer $k$. In one direction there is the conjecture which states that if a finite group $G$ acts freely on a product of $k$ spheres $X=\bbS ^{n_1}\times \cdots \times \bbS ^{n_k}$, then we must have $\rk(G)\leq k$, where $\rk(G)$ denotes the rank of the group $G$, defined as the largest integer $r$ such that $(\bbZ /p)^r \leq G$ for some prime $p$.

In the other direction, there is a conjecture by Benson and Carlson \cite{benson-carlson} in homotopy category which states that if $G$ is a finite group with $\rk (G)\leq k$, then it acts freely on a finite complex $X$ homotopy equivalent to a product of $k$ spheres. The Benson-Carlson conjecture is proved for many groups of small rank, in particular, it is proved to be true for all rank two finite groups which do not involve the group ${\rm Qd}(p)$ for any odd prime $p$ (see \cite{adem-smith}, \cite{jackson1}). For $p$-groups the Benson-Carlson conjecture is known to be true for all $p$-groups with rank $\leq 2$, and for all rank three $p$-groups when $p$ is an odd prime \cite[Theorem 1.1]{klaus}. 

It is shown by Milnor \cite{milnor} that the rank condition $\rk(G) \leq k$ is not sufficient for the existence of a free smooth action on a product of $k$ spheres. He proves, in particular, that the dihedral group $D_{2p}$ of order $2p$, where $p$ is an odd prime, cannot act freely on a manifold which has mod-$2$ homology of a sphere. However, for $p$-groups, there are no known necessary conditions on the group other than the rank condition for constructing free smooth actions. For example, when $G$ is a rank one $p$-group, then $G$ is a cyclic group or a generalized quaternion group, and one can find a unitary representation $V$ of $G$ such that $G$ acts freely and smoothly on the unit sphere $\bbS(V)$.

It is also known that every rank two $p$-group acts freely and smoothly on a product of two spheres. This is proved in \cite[Theorem 1.1]{unlu-yalcin2}, but the construction in this case is much more complicated. The main ingredient in the construction is a theorem of L\" uck and Oliver \cite[Theorem 2.6]{lueck-oliver} which provides a method for constructing $G$-equivariant vector bundles over a given finite dimensional $G$-CW-complex. One of the assumptions of this theorem is the existence of a finite group $\Gamma$ satisfying certain properties. In \cite{unlu-yalcin2}, fusion systems and biset theory were used to show that this finite group $\Gamma$ can be explicitly constructed in that case.  

It is reasonable to ask if the above results for $\rk(G)=1, 2$, holds more generally:

\begin{conjecture}\label{conj:main} Every finite $p$-group $G$ with $\rk(G)=k$ act freely and smoothly on a product of $k$ spheres.
\end{conjecture}

It is clear that this conjecture is true for abelian $p$-groups. More generally, when $G$ is a $p$-group of nilpotency class $\leq 2$, i.e., when $G/Z(G)$ is abelian, then   
the conjecture holds for $G$. This follows from Theorem 1.1 in  \cite{unlu-yalcin3}. In this paper we prove the following theorem which gives further evidence for this conjecture.

\begin{theorem}\label{thm:main} Let $p$ be an odd prime. Then, every rank three $p$-group acts freely and smoothly on a product of three spheres.
\end{theorem}

To prove Theorem \ref{thm:main}, we use a strategy similar to the strategy used in the rank two case. Let $G$ be a rank three $p$-group and let $V=\Ind _{\langle c\rangle}^G W$ denote the complex representation induced from $\la c \ra$, where $c$ is a central element of order $p$ in $G$, and $W$ is a one-dimensional non-trivial representation of $\langle c \rangle$. The   isotropy subgroups $G_x$ of the linear sphere $X=S(V)$ satisfy the property that $G_x \cap \la c \ra =1$. In particular, $\rk(G_x)\leq 2$ for every $x\in X$.  

Let $\cH$ denote the family of all subgroups $H$ of $G$ such that $H \cap \la c\ra =1$. If $\chi: G \to \bbC$ is a class function whose restriction to each $H \in \cH$ is a character, then $\chi$ can be used to define a compatible family of representations $\bV_{\chi}=\{ V_H : H \to U(n) \colon H \in \cH \}$ (see Definition \ref{def:compatible family}).  Moreover, if $\chi$ is an \emph{effective} class function (for every elementary abelian subgroup $E\leq G$ with maximum rank, $\la \chi |_{E}, 1_E \ra =0$), then for every $H\in \cH$, the $H$-action on $\bbS(V_H)$ will have rank one isotropy. It has been shown by M. Klaus in \cite[Proposition 3.3]{klaus} 
that there exists a class function $\chi$ satisfying these properties (this class function was first introduced by M. Jackson in an unpublished work \cite[Proposition 20]{jackson2}). 
 
We apply this method to the class function $\chi$ introduced by Jackson and obtain a compatible family of representations $\bV_{\chi}$. Using this family, we construct a $G$-vector bundle $E\to \bbS(V)$ with fiber type $\bV_{\chi}$. Once this $G$-vector bundle is constructed, we take Whitney sum multiples of this $G$-vector bundle and apply some smoothing techniques to obtain a smooth $G$-action on a product of two spheres $M=\bbS(V)\times \bbS ^m$ with rank one isotropy. Finally we apply \cite[Theorem 6.7]{unlu-yalcin2} to $M$ and obtain a free smooth $G$-action on a product of three spheres $\bbS(V) \times \bbS^m \times \bbS^k$ for some $m,k\geq 1$.  
 
The key step in this construction is the construction of a $G$-vector bundle over $\bbS(V)$ with fiber type $\bV_{\chi}$. For this step we use a generalization of the L\" uck-Oliver theorem  on constructions of $G$-vector bundles
(see Theorem \ref{thm:G-bundleconst}). The main assumption of the L\" uck-Oliver  theorem is that the given compatible family of representations factors through a finite group $\Gamma$ (see Definition \ref{def:factors through G}).  However, we were not able to find such a finite group $\Gamma$ for the family $\bV_{\chi}$. 

On the other hand, it is possible to find a collection of subfamilies $\{ \cH_d \}$ which covers $\cH$ such that the restriction of $\bV_{\chi}$ to $\cH_d$ factors through a finite group $\Gamma _d$. So we prove a theorem (Theorem \ref{thm:G-bundleconst}) which has the same conclusion as the L\" uck-Oliver theorem but it works under a weaker assumption that the given compatible family of representations factors through a diagram of finite subgroups satisfying certain connectedness properties. Using this theorem, we are able to do the $G$-vector bundle construction for the family $V_{\chi}$  and complete the proof of Theorem \ref{thm:main}.

The paper is organized as follows. In Section \ref{sect:G-bundles} we introduce necessary definitions and state the L\" uck-Oliver theorem mentioned above. Section \ref{sect:proof of Thm 2.10} is devoted to the proof Theorem \ref{thm:G-bundleconst} which is a generalization of the L\" uck-Oliver theorem. In Section \ref{sect:FreeActions}, we prove some consequences of Theorem \ref{thm:G-bundleconst}. In Section \ref{sect:RankThreeConst}, we  prove Theorem \ref{thm:main} using the strategy described above.

{\bf Acknowledgements:} This paper was completed when the author was visiting McMaster University during Fall 2014. The author thanks Ian Hambleton and McMaster University for hospitality and financial support. This research is also supported by the Scientific and Technological Research Council of Turkey (T\" UB\. ITAK) through the research program B\. IDEB-2219. I thank the referee for careful reading of the paper and for the corrections.

\section{Constructing $G$-vector bundles}\label{sect:G-bundles}

Let $G$ be a finite group and $X$ be a $G$-CW-complex. A $G$-vector bundle over $X$ is a vector bundle $p: E \to X$ such that $p$ is a $G$-map and $G$ acts on $E$ via bundle isomorphisms. Note that for each $x\in X$, there is an action of isotropy subgroup $G_x$ on the fiber space $V_x=p^{-1}(x)$ which is a vector space and the action of $G_x$ on $V_x$ is linear. 

Let $\cH$ be a family of subgroups of $G$. \emph{Throughout the paper ``a family of subgroups" always means that it is a set of subgroups of $G$ which is closed under conjugation and taking subgroups}. Let $\bV=\{ V_H\}_{H \in \cH}$ be a collection of $H$-representations over the family $\cH$. We say that the $G$-vector bundle $p:E\to X$ has  \emph{fiber type} $\bV$ if for every $x\in X$, the isotropy subgroup $G_x$ lies in the family $\cH$ and there is an isomorphism  of $G_x$-representations $V_x \cong V_{G_x}$. Note that the collection of representations $\{V_H\}$ arising as fibers of a $G$-vector bundle satisfies the following compatibility condition.

\begin{definition}\label{def:compatible family} Let $G$ be a finite group and $\cH$ be a family of subgroups of $G$. A collection of representations $\bV=(V_H)_{H \in \cH}$ is called a {\it compatible family} if for every map $c_g: H\to K$ defined by $c_g (h)=ghg^{-1}$, where $g \in G$ and $H, K\in \cH$, there is a $H$-vector space isomorphism $V_H \cong(c_g)^* (V_K)$.
\end{definition}

In \cite{lueck-oliver}, L\" uck and Oliver consider the question of constructing a $G$-vector bundle $q:E \to X$ over a given finite dimensional $G$-CW-complex $X$, such that the fiber type of $q$ is the given compatible family $\bV$. They observe that in general these $G$-vector bundles may not exist, but they also proved that if $\bV$ factors through a finite group, then one can construct a $G$-vector bundle over $X$ with fiber type $\bV ^{\oplus k}$ for some positive integer $k$ (see \cite[Theorem 2.6]{lueck-oliver}). This theorem is the main tool for constructing smooth actions on products of spheres given in \cite{unlu-yalcin2}. Before we state this theorem, we first introduce some necessary definitions.  
 
Let $\Gamma$ be a compact Lie group. A $G$-equivariant principal $\Gamma$-bundle over a $G$-CW-complex $X$ is a principal $\Gamma$-bundle $p:E\to X$ 
such that $p$ is a $G$-map between left $G$-spaces
and the left $G$-action on $E$ commutes with the right $\Gamma$-action.
Note that as in the $G$-vector bundle case, for each $x\in X$, there is a $G_x$-action on the fiber space $p^{-1} (x)$. The fiber space $p^{-1}(x)$ is a free $\Gamma$-orbit $e\cdot \Gamma$ for some $e\in E$ such that $p(e)=x$.  This gives a homomorphism $\alpha _{G_x} : G_x \to \Gamma$ defined by $\alpha _{G_x} (h)=\gamma $ for $h\in G_x$, where $\gamma \in \Gamma$ is the unique element in $\Gamma$ such that $he=e\gamma$. Note that this homomorphism is well-defined up to a choice of the element $e\in p^{-1} (x)$, so it defines 
an element in $\Rep (G_x , \Gamma):=\Hom (H, \Gamma )/\Inn(\Gamma)$ where $\Inn(\Gamma)$ denotes the group of conjugation actions of $\Gamma$ on itself.

\begin{definition} Let $G$ be a finite group and $\cH$ be a family of subgroups of $G$. A collection of representations $\bA =(\alpha _H : H \to \Gamma)_{H\in \cH}$ over $\cH$ is called a {\it compatible family} if for every map $c_g: H \to K$ induced by conjugation $c_g(h)=ghg^{-1}$, where $g\in G$ and $H,K\in \cH$, there exists a $\gamma\in \Gamma$ such that the following diagram commutes:
$$\xymatrix{
H \ar[d]^{c_g} \ar[r]^{\alpha _H}
& \Gamma \ar[d]^{ c_{\gamma }}  \\
K  \ar[r]^{\alpha _K}
&  \Gamma}\\ $$
This is equivalent to saying that $\bA =(\alpha _H ) _{H \in \cH} $ is an element of the limit $$\lim _{G/H \in \Or _{\cH} G} \Rep (H , \Gamma )$$
where $\Or _{\cH } G$ denotes the orbit category of $G$ over the family $\cH$. Recall that the orbit category $\Or _{\cH} G$ is the category whose objects are transitive $G$-sets $G/H$ with $H \in \cH$ and whose morphisms are given by $G$-maps $\Map _G (G/H, G/K)$.
\end{definition}

\begin{definition}\label{def:factors through G} Let $\bV$ be a compatible family of unitary representations over a family of subgroups $\cH$. We say that $\bV$ \emph{factors through a finite group} $\Gamma$ if there exists a triple $(\Gamma, \rho, \bA)$, where $\Gamma$ is a finite group, $\rho : \Gamma\to U(n)$ is a unitary representation of $\Gamma$, and  $\bA=(\alpha _H: H \to \Gamma)_{H\in \cH}$ is a compatible family of representations, such that $\bV=\rho \circ \bA$.  
\end{definition}

Now we state the L\"uck-Oliver theorem mentioned in the introduction.

\begin{theorem}[see Theorem 2.6 in \cite{lueck-oliver}]\label{thm:lueck-oliver} Let $G$ be a finite group and $\cH$ be a family of subgroups in $G$. Let $X$ be a finite dimensional $G$-CW-complex with isotropy subgroups in $\cH$. Suppose that we are given  a compatible family $\bV$ of unitary representations  over $\cH$ and that $\bV$ factors through a finite group $\Gamma$. Then there is an integer $k\geq 1$ such that there exists a $G$-vector bundle $E \to X$ with fiber type $\bV ^{\oplus k}$.
\end{theorem}

We are interested in proving a generalization of Theorem \ref{thm:lueck-oliver}. We will show that the conclusion of this theorem still holds under the weaker assumption that $\bV$ factors through a diagram of finite groups instead of a single finite group $\Gamma$. We now introduce the necessary terminology to explain exactly what we mean by this.

Let $\cD$ be a finite poset considered as a category. Note that in $\cD$, there is a unique morphism between two objects $x, y \in \cD$ if and only if $x\leq y$.  Later we will assume that $\cD$ is a one-dimensional poset category. This means that if $x \leq y \leq z$ is a chain in $\cD$ then either $x=y$ or $y=z$. When $\cD$ is one-dimensional, the set of objects in $\cD$ can be written as a disjoint union ${\rm obj} (\cD)=D_1\amalg D_2$ where  if $x< y $ in $\cD$ then $x\in D_1$ and $y\in D_2$. Here $x<y$ means that $x\leq y$ but $x\neq y$. Sometimes these posets are called bipartite posets.

\begin{definition}\label{def:diagram of groups} Let $\cD$ be a finite poset category. 
\begin{enumerate}
\item
A {\it diagram of groups} $\Gamma _{*}$ over $\cD$ is a functor from $\cD$ to the category of groups. We denote the group associated to $d \in \cD$ by $\Gamma _d$ and for each $x\leq y$, the corresponding group homomorphism is denoted by $\mu _{x,y}\colon \Gamma _x \to \Gamma _y$.  We say $\Gamma _*$ is a {\it diagram of finite groups} if for all $d\in \cD$, the groups $\Gamma_d$ are finite.

\item Let $n$ be a fixed positive integer. 
A {\it diagram of  representations} of $\Gamma_*$ of degree $n$ is a collection of homomorphisms
$\rho_d \colon \Gamma _d\to U(n)$, one for each $d\in \cD$, such that for every $x, y$ in $\cD$ with $x\leq y$, the representations $\rho _x$ and $\rho _y \circ \mu _{x,y}$ are isomorphic. 

\item Let $\cH$ be a family of subgroups of $G$ and $\{\cH_d\}_{d\in \cD}$ be a collection of subfamilies of $\cH$ (for each $d\in \cD$, $\cH_d$ is closed under conjugation and taking subgroups). If for every $x\leq y$ in $\cD$, $\cH_x \subseteq \cH_y$, then we call $\{\cH_d \} _{d\in \cD}$ a \emph{diagram of subfamilies of $\cH$} over $\cD$ and denote it by $\cH _*$. A diagram of subfamilies $\cH_*$ can also be thought as a functor from $\cD$ to the poset of subfamilies of $\cH$.
\end{enumerate}
\end{definition}

\begin{remark} In our applications, the maps $\mu _{x,y}: \Gamma _x \to \Gamma _y$ are always injective, but we do not assume this in the definition of a diagram of groups. In particular, Theorem \ref{thm:G-bundleconst} and Theorem \ref{thm:AlmostConnected} hold for the maps $\mu_{x,y}$ which are not necessarily injective.
\end{remark}

We do not assume that the subfamilies $\cH_d$ cover $\cH$ in the definition but we have a connectedness assumption which implies that $\cup _{d\in \cD} \cH_d =\cH$.

\begin{definition}\label{def:connectedness}
Let $\cD$ be a one-dimensional poset category and $\cH_*$ be a diagram of subfamilies of $\cH$ over $\cD$. For each $H \in \cH$, let $\cD_H$ denote the full subposet $\{ d\in \cD \, | \, H \in \cH_d\}$. We say $\cH_*$ is {\it strongly connected} if for every $H \in \cH$, the realization of $\cD _H$ is simply connected (i.e., non-empty, connected, and having trivial fundamental group).  
\end{definition}

Next, we define what we mean by a diagram of compatible family of representations:

\begin{definition}\label{def:compatible over d} Let $\cH_*$ be a diagram of subfamilies and $\Gamma _*$ be a diagram of groups over a finite poset $\cD$. Suppose that for each $d\in \cD$, we are given a compatible family of representations
$$\bA _d=\{\alpha _H ^d: H \to \Gamma _d \mid H \in \cH_d \}.$$ 
We say $\bA_*=(\bA _d)_{d\in \cD}$ is a {\it diagram of compatible families of representations} if it satisfies the condition that for every $x\leq y $ in $\cD$, the restriction of $\bA _y$ to $\cH _x$ is equal to $\mu _{x,y}\circ \bA _x$. We write this condition as $\bA _y |_{\cH _x}=\mu_{x,y} \circ \bA _x$ for all $x\leq y$.
\end{definition}

\begin{remark} Note that another way to define this compatibility condition is to require that for every map $c_g: H \to K$ induced by conjugation $c_g(h)=ghg^{-1}$, where $g\in G$, and for every $x\leq y$ in $\cD$ such that $H\in \cH_x$ and $K \in \cH_y$,  there exists a $\gamma\in \Gamma _y$ such that the following diagram commutes:
$$\xymatrix{
H \ar[d]^{c_g} \ar[r]^{\alpha ^{x} _H}
& \Gamma_x \ar[d]^{ c_{\gamma }\circ \mu _{x, y}}  \\
K  \ar[r]^{\alpha ^y _K}
&  \Gamma _y}\\ $$
If we take $x=y=d$ in the above diagram, we obtain that the family $\bA_d=(\alpha ^d _H ) _{H\in \cH _d}$ is a compatible family of representations $\alpha _H : H \to \Gamma _d$ over $\cH_d$ in the usual sense. If we take $x< y$ in $\cD$, the commutativity of the diagram above is equivalent to the condition $\bA _y |_{\cH _x}=\mu_{x,y} \circ \bA _x$.
\end{remark}

Now we explain what we mean when we say a family of representations factors through a diagram of finite groups.
 
\begin{definition}\label{def:factors through G_*}
Let $\bV=(V_H)_{H\in \cH}$ be a compatible family of unitary representations over a family of subgroups $\cH$. We say that $\bV$ {\it factors through a diagram of finite groups} $\Gamma _*$ if there exists a quadruple $(\Gamma _*, \rho_*, \cH_*, \bA _*)$, where
\begin{enumerate}
\item $\Gamma_*$ is a diagram of finite groups over a finite poset category $\cD$, 
\item $\rho_*$ is a representation of $\Gamma _*$, 
\item $\cH_*$ is a diagram of subfamilies over $\cD$, and  
\item $\bA_*=(\bA _d)_{d\in \cD}$ is a diagram of compatible families of representations defined over $\cH_*$, 
\end{enumerate}
such that for each $d\in \cD$, the equality $\bV|_{\cH _d}=\rho_d \circ \bA_d$ holds. \end{definition}

Finally we define the main assumption in our theorems.

\begin{definition}\label{def:strongly connected G_*}
Let $\bV=(V_H)_{H\in \cH}$ be a compatible family of unitary representations over a family of subgroups $\cH$. Suppose that $\bV$ factors through a diagram of finite groups $\Gamma _*$ over a one-dimensional diagram $\cD$. If $\cH_*$ is strongly connected, then we say $\bV$ factors through a \emph{strongly connected one-dimensional diagram of finite groups $\Gamma _*$}.  
\end{definition}

\section{A generalization of the L\" uck-Oliver theorem}\label{sect:proof of Thm 2.10}

The main aim of this section is to prove the following theorem.

\begin{theorem}\label{thm:G-bundleconst} Let $G$ be a finite group, $\cH$ be a family of subgroups of $G$, and $X$ be a finite dimensional $G$-CW-complex with isotropy subgroups in $\cH$. 

Suppose that we are given a compatible family $\bV$ of unitary representations over $\cH$ which factors through a strongly connected one-dimensional diagram of finite groups $\Gamma _*$. 

Then, there is a positive integer $k$ such that there exists a $G$-vector bundle $E \to X$ with fiber type $\bV ^{\oplus k}$.
\end{theorem}

The proof is obtained by modifying the proof of \cite[Theorem 2.6]{lueck-oliver}. We will use the notation introduced in \cite[Section 2]{lueck-oliver}. In particular, throughout $B _{\cH} (G, \bV)$ denotes the classifying space of $G$-vector bundles with fiber type $\bV$. Similarly, for each $d\in \cD$, $B_{\cH _d} (G, \bA _d)$ denotes the classifying space of $G$-equivariant principal $\Gamma _d$-bundles with fiber type $\bA _d$. For each $d\in \cD$, we can use the representation $\rho _d: \Gamma _d \to U(n)$ to convert a $G$-equivariant principal $\Gamma_d$-bundle $q \colon E \to X$ to a $G$-vector bundle $\widetilde q: E \times _{\Gamma _d} V \to X$ where $V$ denotes $\Gamma _d$-vector space defined by the representation $\rho _d$. Applying this construction to the universal principal $\Gamma _d$-bundle over $B_{\cH_d} (G, \bA _d )$, we get a map $$B\rho_d\colon B_{\cH _d} (G, \bA _d )\to B_{\cH } (G, \bV)$$ for each $d\in \cD$ as the classifying map of the $G$-vector bundle obtained by the above construction. 

A similar argument can be used to show that for every non-identity map $x \to y$ in $\cD$, there is a map $B\mu _{x,y} : B_{\cH _x} (G, \bA _x) \to B_{\cH _y } (G, \bA _y )$ defined by converting the universal $G$-equivariant principal $\Gamma _x$-bundle to a $\Gamma _y$-bundle via the homomorphism $\mu _{x,y} : \Gamma _x \to \Gamma_y$. For this to work one needs the equality $\bA _y |_{\cH _x} =\mu _{x,y}\circ \bA _x$ to hold which we have by the compatibility assumption on $(\bA _d )_{d\in \cD}$ described in Definition \ref{def:compatible over d}. Note that since $\cD$ is a one-dimensional category, the assignment $d\to B_{\cH _d } (G, \bA _d)$ together with the assignment $\mu _{x,y}\to B\mu _{x,y}$ defines a functor $F$ from $\cD$ to the category of topological spaces. 

Let $Y:=\hocolim_{\cD} F$ denote the homotopy colimit of the functor $F : \cD \to {\rm Top}$ (see \cite[Section 4.5]{benson-smith} for more details on homotopy colimits). Since $\cD$ is a one-dimensional category, $Y$ can be described as the identification space $$\hocolim _{\cD} F=\Bigl \{ \Bigl ( \coprod _{d\in \cD} B_{\cH _d } (G, \bA _d) \Bigr )\amalg \Bigl ( \coprod _{x< y } B_{\cH _x } (G, \bA _x)\times [0,1] \Bigr ) \Bigr \} \Big /\sim$$ 
where $B_{\cH _x} (G, \bA _x)\times \{ 0\} $ is identified with $B_{\cH_x} (G, \bA _x)$ via the identity map, and on the other end $B_{\cH _x} (G, \bA _x)\times \{ 1\} $ is identified with $B_{\cH _y} (G, \bA _y)$ via the map $B\mu _{x,y}$. 

For every $H \in \cH$, the fixed point set $Y^H$ is nonempty if and only if $H \in \cH_d$ for some $d\in \cD$. Since $\cH_*$ is strongly closed, we have $\cup _{d\in \cD} \cH_d=\cH$, hence we can conclude that for every $H \in \cH$, we have $Y^H \neq \emptyset$. We also have the following:

\begin{lemma}\label{lem:cohcalc} For every $H \in \cH$, the reduced homology group $\widetilde H_j (Y^H)$ has finite exponent for all $j$.  
\end{lemma}

\begin{proof} Take $H \in \cH$. The fixed point subspace $Y^H$ is the homotopy colimit of the functor $$F^H : d \to B_{\cH_d} (G, \bA_d)^H.$$ The fixed point subspace $B_{\cH_d} (G, \bA_d)^H$ is nonempty if and only if $H \in \cH_d$. So the space $Y^H$ can be considered a homotopy colimit of the functor $F^H$ over the subposet $\cD_H$ generated by $\{ d \in \cD : H\in \cH_ d\}$.  It is shown in \cite[Lemma 2.4]{lueck-oliver} that for each $d\in \cD$, the fixed point space $B_{\cH_d}(G, \bA_d)^H$ is homotopy equivalent to the classifying space $BC_{\Gamma _{d}} (\alpha ^{d} _H)$ where $C_{\Gamma _{d}} (\alpha ^{d} _H)$ denotes the centralizer of $\alpha ^d _H (H) $ in $\Gamma _d$. Since $\Gamma _d$ is a finite group, the reduced homology group of $C_{\Gamma _{d}} (\alpha ^{d} _H)$ has finite exponent, hence $\widetilde H_t (B_{\cH_d}(G, \bA_d)^H)$ has finite exponent for all $d\in \cD$ and for all $t\geq 0$. 

To calculate the homology groups of $Y^H=\hocolim _{\cD _H } F^H$, we use the Bousfield-Kan homology spectral sequence (see \cite[Theorem 4.8.7]{benson-smith}). In this case, this spectral sequence takes the form $$E^2 _{s,t}=\colim _s H_t (B_{\cH_d} (G, \bA _d )^H) \Rightarrow H_{s+t} (Y^H )$$ where the colimit is over the category $\cD_H$. At this point it is useful to consider all the cohomology groups with coefficients in rational numbers. By the above observation for all $H \in \cH$, we have $H_t (B_{\cH_d} (G, \bA _d )^H, \bbQ )\cong H_t (pt , \bbQ)$ for all $t\geq 0$. So we obtain that  $$H_j (Y^H ; \bbQ ) \cong \colim _j H_0(pt, \bbQ) \cong  H_j (|\cD_H |; \bbQ)$$ for every $j \geq 0$, where $|\cD_H|$ denotes the realization of the poset $\cD_H$.   Since $\cD$ is one-dimensional and $\cH_*$ is strongly connected, for every $H \in \cH$, we have $\widetilde H_j (|\cD_H |; \bbZ)=0$ for all $j$. Hence the proof of the lemma is complete.
\end{proof}

Now we show how the proof of Theorem \ref{thm:G-bundleconst} can be completed using Lemma \ref{lem:cohcalc}. Note that for every $x\leq y$ in $\cD$, the representations $\rho _x$ and $\rho _y \circ \mu _{x,y}$ are isomorphic, hence the maps $B\rho _x$ and $B\rho _y \circ B\mu _{x,y}$ are homotopic. Using these homotopies we can extend the $G$-maps $B\rho _d: B_{\cH _d} (G, \bA _d)\to B_{\cH } (G, \bV)$ to a $G$-map $B\rho_* : Y \to B_{\cH } (G, \bV)$. 

The isotropy subgroups of $Y$ are in $\cH$, so there is also a $G$-map from $Y$ to the universal space $E_{\cH} G$ for the family $\cH$ (see \cite[Definition 2.1]{lueck-oliver}). Let us denote this map by $\beta :Y \to E_{\cH}G$. Let $Z$ denote the mapping cylinder of $\beta$. For every positive integer $k$, we have a $G$-map $f_k: Y \to  B_{\cH} (G, \bV ^{\oplus k} )$ obtained as the composition $$ f_k: Y \maprt{B\rho _*} B_{\cH } (G, \bV) \maprt{w_k} B_{\cH} (G, \bV ^{\oplus k} )$$ where the second map is the map induced by Whitney sum construction on $G$-vector bundles.

We want to show that for every positive integer $n$, there is a positive integer $k$ such that $f_k$ can be extended to a $G$-map $$\tilde f_k ^{(n)}: Z^{(n)}\cup Y \to B_{\cH} (G, \bV ^{\oplus k} ),$$ where $Z^{(n)}$ denotes the $n$-skeleton 
of $Z$. Observe that this finishes the proof of Theorem \ref{thm:G-bundleconst} because given a finite dimensional $G$-CW-complex $X$ with isotropy set $\cH$, there is a $G$-map from $X$ to $E_{\cH} G ^{(n)}$ for some $n$. Then composing this map with $\tilde f _k ^{(n)}$ we get a $G$-map $\tilde f_k ^X: X\to B_{\cH } (G, \bV ^{\oplus k} )$. The desired $G$-vector bundle over $X$ is the one obtained by pulling back the universal bundle over $B_{\cH } (G, \bV ^{\oplus k} )$ via $\tilde f_k ^X$. The details of this argument can be found in the proof of \cite[Theorem 2.6]{lueck-oliver}.

To show that for every $n\geq 0$, there is an integer $k$ such that $f_k$ can be extended to $\tilde f_k ^{(n)}: Z^{(n)}\cup Y \to B_{\cH} (G, \bV ^{\oplus k} )$, we   first observe that  $\tilde f_1 ^{(2)}$ exists since $B_{\cH } (G, \bV )^H $ is simply connected for all $H\in \cH$. Now assume that for some $n\geq 2$ there exists a $k\geq 1$ such that the map $f_k$ has been extended to $\tilde f^{(n)}_k$. We will show that by replacing $k$ with its multiple if necessary, we can extend $\tilde f^{(n)}_k$ to a map $\tilde f^{(n+1)} _{k}$ defined on $ Z^{(n+1)}\cup Y$. For this we use equivariant obstruction theory. 

Note that the obstructions for lifting $\tilde f_k ^{(n)}$ to $\tilde f_k ^{(n+1)} $ lies in the Bredon cohomology group $$H^{n+1} _G( Z, Y ; \pi _{n} (B_{\cH } (G, \bV ^{\oplus k} )^?)).$$ If these obstructions have finite exponent then they can be killed by taking further Whitney sums, i.e., by making $k$ bigger (see \cite[Theorem 2.6]{lueck-oliver} for details of this argument). So the proof is complete if we show that the above cohomology groups have finite exponent for all $n\geq 2$. Note that these cohomology groups are Bredon cohomology groups of the pair $(Z,Y)$ with coefficients in a local coefficient system, defined by $G/H \to \pi _{n} (B_{\cH } (G, \bV ^{\oplus k} )^H)$. Recall that a coefficient system over the family $\cH$ is a module over the orbit category $\Gamma_G:=\Or _{\cH} G$. So to complete the proof of  Theorem \ref{thm:G-bundleconst}, it is enough to prove the following proposition.

\begin{proposition}\label{pro:cohcalc}
Let $Z$ and $Y$ be as above and $M$ be an arbitrary $\bbZ \Gamma_G$-module. Then,
the Bredon cohomology group $H^{n+1} _G ( Z, Y ; M )$ has finite exponent for all $n\geq 2$.
\end{proposition}

\begin{proof} The Bredon cohomology of a pair can be calculated using an hyper-cohomology spectral sequence with $E_2$-term $$E_2^{p,q}=\Ext ^p _{\bbZ \Gamma_G} (H_q (Z^?, Y^?), M)$$ which converges to the equivariant cohomology group $H^{p+q} _G (Z, Y ; M)$ (see \cite[Proposition 3.3]{unlu-yalcin2}). Hence to show that the cohomology groups $H^{n+1} _G (Z, Y ; M)$ have finite exponent for all $n\geq 2$, it is enough to show that the ext-groups $$\Ext ^p _{\bbZ \Gamma_G} (H_q (Z^?, Y^?), M)$$
are finite groups for all $p,q$ with $p+q \geq 3$.  

We have that $Z^H \simeq (E_{\cH } G )^H \simeq *$ for every $H \in \cH$. So, we can conclude that $H_i (Z^H, Y^H )\cong \widetilde H_{i-1} (Y^H)$ for all $i \geq 1$ and  $H_0 (Z^H, Y^H)\cong  \bbZ $ if $Y^H= \emptyset$ and zero otherwise. Since $Y^H \neq \emptyset$ for every $H \in \cH$, we have $H_0 (Z^H, Y^H)=0$ for every $H \in \cH$. Moreover, by Lemma \ref{lem:cohcalc}, $\widetilde H_{i-1} (Y^H)$ has finite exponent for every $i\geq 1$. Hence the proof is complete.
\end{proof}

\section{Construction of free actions on products of spheres}\label{sect:FreeActions}

In this section we prove two consequences of Theorem \ref{thm:G-bundleconst} which are going to be main tools for the constructions of free actions on products of spheres. Throughout the section when we say $M$ is a smooth $G$-manifold we always mean that $M$ is a smooth manifold with a smooth $G$-action.  

\begin{theorem}\label{thm:G-actionconst} 
Let $G$ be a finite group and $\cH$ be a family of subgroups of $G$. Let  $M$ be a finite dimensional smooth $G$-manifold with isotropy subgroups lying in $\cH$. 
 
Suppose that we are given a compatible family $\bV$ of unitary representations over $\cH$ which factors through a strongly connected one-dimensional diagram of finite groups $\Gamma _*$.

Then, there exists a smooth $G$-manifold $M'$ diffeomorphic to $M\times \bbS^m$ for some $m>0$ such that for every $x\in M$, the $G_x$-action on $\{x \}\times \bbS^m$ is diffeomorphic to the linear $G$-sphere $\bbS (V_{G_x} ^{\oplus k} )$ for some $k\geq 1$. 
\end{theorem}

\begin{proof} The proof is essentially the same as the proof of Corollary 4.4 in \cite{unlu-yalcin2}. We summarize the argument here for the convenience of the reader. By Theorem \ref{thm:G-bundleconst} there is a topological $G$-vector bundle $p: E \to M$ with fiber type $\bV ^{\oplus k}$ for some $k\geq 1$. This bundle is obtained as a pullback of a bundle 
over $E_{\cH}G ^{(n)}$ for some $n$. By taking the value of $n$ larger than the dimension of $M$, we can assume that the bundle $p: E \to M$ is non-equivariantly a trivial bundle. Note that here we use the fact that $\cH$ is closed under taking subgroups, in particular, we have $1\in \cH$, hence $E_{\cH}G$ is contractible.

As a $G$-vector bundle, the bundle $p:E \to M$ is equivalent to a smooth $G$-vector bundle $p': E' \to M$. This smooth $G$-bundle can be constructed by replacing the universal $G$-bundle with a smooth universal $G$-bundle (see the proof of Corollary 4.4 in \cite{unlu-yalcin2} for details). Since $p$ is non-equivariantly trivial, the bundle $p'$ is also non-equivalently trivial as a topological bundle. One can replace continuous trivialization with a smooth trivialization to obtain a diffeomorphism $\bbS (E')\approx M \times \bbS ^m$ where $\bbS(E')$ is the total space of the sphere bundle $\bbS (E')\to M$  associated to $p$. For every $x\in M$, the sphere $\{x \} \times \bbS^m$ is mapped to $\bbS ((p')^{-1} (x)) \subseteq \bbS (E')$ under the above diffeomorphism. The $G_x$-action on $(p')^{-1} (x)$ is isomorphic to $G_x$-action on $p^{-1} (x)$ as $G_x$-vector spaces. Since $p: E \to M$ has fiber type $\bV^{\oplus k}$, the $G_x$-action on $p^{-1}(x)$ is isomorphic to $V_{G_x} ^{\oplus k}$. Thus we can conclude that $G_x$-action on $\{x \} \times \bbS^m$ is diffeomorphic to $G_x$-action on $\bbS (V_{G_x} ^{\oplus k})$ for some $k \geq 1$. 
\end{proof}

As an application of Theorem \ref{thm:G-actionconst}, we prove the following result which is a slight generalization of \cite[Theorem 6.7]{unlu-yalcin2}.

\begin{theorem}\label{thm:rank1isotropy}
Let $G$ be a finite group acting smoothly on a manifold $M$ such that all isotropy subgroups $G_x$ are rank one subgroups with prime power order. Then, there exists a positive integer $N$ such that $G$
acts freely and smoothly on $M \times \bbS^N$.
\end{theorem}

\begin{proof} Let $\cH$ denote the family of all rank one subgroups of $G$ with prime power order, plus the trivial subgroup.  If $H$ is a rank one $p$-group, then it has a unique subgroup of order $p$, denoted by $\Omega _1 (H)$. Let $\cD$ denote the poset of conjugacy class representatives of subgroups $K \leq G$ such that either $K$ has prime order or $K=1$. The ordering in $\cD$ is given by the usual inclusion of trivial subgroup into other subgroups, hence the realization of $\cD$ is a star shaped tree. For every $1\neq d\in \cD$, let $\cH_d$ denote the subfamily $$\cH_d:=\{ H \in \cH \colon \Omega_1(H) \simeq _G d \}\cup \{ 1\}.$$ Take $\cH_{1}=\{1\}$. It is easy to see  that the collection of subfamilies $\{ \cH _d \} _{d\in \cD}$ covers $\cH$ and that $\cH_*$ is strongly closed. 

For each $1\neq d\in \cD$, take $\Gamma _d=N_G(d)$, normalizer of the subgroup $d$ in $G$, and let $\Gamma _1=\{1\}$. For every $d\in \cD$, let $m_d =|N_G(d)|(p-1)/p$ where $p$ is equal to the order of the subgroup $d$. Let $n$ be a positive integer that is divisible by $m_d$ for all $d\in \cD$, and let $n_d=n/m_d$. For each $1\neq d \in \cD$, let $\rho_d: \Gamma _d \to U(n)$ be a $n_d$ multiple of the induced representation $V_d=\Ind ^{N_G(d)}_d W$ where $W : d \to U(p-1) $ is  the reduced regular representation of $d$. We take $\rho_1: \Gamma _1 \to U(n)$ as $n$ copies of the trivial representation of the trivial group. It is clear that the family $\{\rho_d\}$ is a representation of the diagram of groups $\Gamma_*$. 

Now we describe the diagram $\bA_*$ of compatible families of representations. For each $1\neq d\in \cD$, and $H \in \cH_d$, let $\alpha_H ^d:H \to \Gamma _d$ be the map defined by $h \to ghg^{-1}$ where $g$ is an element in $G$ such that $g\Omega_1( H)g^{-1}=d$. Note that the choice of $g$ is unique up to an element in $\Gamma _d=N_G(d)$, so $\alpha_H^d$ is well-defined as an element in $\Rep(H, \Gamma _d)=\Hom(H, \Gamma _d )/\Inn (\Gamma_d)$. For $d=1$, we take $\alpha_1:1 \to \Gamma _1$ as the identity map. 

Let $\bV$ be the compatible family of representations $V_H : H \to U(n)$ over $H \in \cH$ such that for all $H \in \cH_d$, $V_H=\rho_d \circ \alpha^d_H$. The family $\bV$ satisfies the conditions of Theorem \ref{thm:G-actionconst}, so by applying this theorem, we obtain a smooth $G$-manifold $M'$ diffeomorphic to $M \times S^N$ for some $N\geq 1$. Since all the representations $V_H$ in the family $\bV$ are free, the $G$-action on $M'$ is free.
\end{proof}

Now we will prove a slightly stronger version of Theorem \ref{thm:G-actionconst} which will be used  in the next section for the construction of free actions of rank three $p$-groups. We first prove a lemma.

\begin{lemma}\label{lem:ZeroExt} Let $G$ be a finite group, $\cH$ be a family of subgroups of $G$, and let $\Gamma _G :=\Or _{\cH} (G)$ denote the orbit category of $G$ over $\cH$. Suppose that $N$ is a $\bbQ \Gamma _G$-module such that $N(H)= 0$ for all $H\in \cH$ except possibly when $H$ is a cyclic subgroup of prime power order. Then for every $\bbQ \Gamma _G$-module $M$, we have $\Ext ^i _{\bbQ \Gamma _G } (N, M)=0$ for all $i\geq 2$.
\end{lemma}

\begin{proof} The statement is equivalent to the statement that $N$ has a projective resolution of the form $0\to P_1 \to P_0 \to N\to 0$ as a $\bbQ \Gamma _G$-module. Note that we only need to prove this for an atomic functor and the general case follows by induction on the length of the module $N$. Recall that a $\bbQ \Gamma _G$-module $N$ is called an atomic functor if it has nonzero value only on conjugacy classes of a fixed subgroup $H$. In this case $N=I_H A$ for some rational $W_G(H)$-module $A$, where $W_G(H)=N_G(H)/H$ and $I_H $ denote the inclusion functor (see \cite[9.29]{lueck}) defined by $$(I_H A)(K) =\begin{cases} A \otimes _{\bbQ W_G(H)} \bbQ \Map _G (G/K, G/H) & \text{if $H=_G K$}\\
0 & \text{otherwise}\\
\end{cases}$$   If $H=1$, then $I_1A$ is a projective $\bbQ \Gamma _G$-module. So assume $H\neq 1$. For $P_0$ we will take $E_HA$, where $E_H$ denotes the extension functor defined by $$(E_H A)(K) =A \otimes _{\bbQ W_G(H)} \bbQ \Map _G (G/K, G/H)$$
for  $K \in \cH$ (see \cite[9.28]{lueck}). Since $E_H$ takes projective $\bbQ W_GH$-modules to projective $\bbQ \Gamma_G$-modules, $E_HA$ is projective and there is a  canonical map $E_H A \to I_H A$ which comes from adjointness properties of the functor $E_H$. Let $X_HA$ denote the kernel of this map.  Then $$(X_H A) (L)=A \otimes _{W_G H} \bbQ \Map_G (G/L, G/H)$$ for $L < _G H$ and $(X_HA)(L)=0$ for  all other subgroups $L \leq G$. There are obvious restriction and conjugation maps between nonzero values of $X_HA$ induced by $G$-maps $G/L\to G/L'$.  

Let $H$ be a cyclic group of order $p^n$ for some $n\geq 1$, and $K$ be an index $p$ subgroup in $H$. We claim that $X_HA \cong E_K \bigl ((X_HA)(K) \bigr )$. Note that this will imply that $X_HA$ is a projective $\bbQ \Gamma _G$-module, hence we will have the desired projective resolution.

To show the claim, observe that there is a natural map $$\varphi: E_K \bigl ( (X_H A) (K) \bigr )\to X_H A$$
which induces an isomorphism at subgroups conjugate to $K$. When evaluated at $L\leq K$, this map gives a map of $W_G L$-modules
$$ A \otimes _{W_G H} \bbQ \Map_G (G/K, G/H)\otimes _{W_GK} \bbQ \Map _G (G/L, G/K) \to  A \otimes _{W_G H} \bbQ \Map_G (G/L, G/H).$$
which is induced by a map of $W_GH$-$W_GL$-bisets 
$$ \mu: \Map_G (G/K, G/H)\times _{W_GK} \Map _G (G/L, G/K) \to  \Map_G (G/L, G/H).$$ Note that $\mu$ takes the equivalence class of a pair of maps $(f_1,f_2)$ to their composition $f_1\circ f_2$. We claim that $\mu$ is a bijection for all $L\leq K$. This will imply that $\varphi$ is an isomorphism. 

Note that a $G$-map $f : G/L\to G/H$ is uniquely determined by a coset $gH$ where $f(L)=gH$. For this to make sense, the coset representative $g$ has to satisfy the condition that $g^{-1} L g \leq H$. In other words, we can identify $\Map _G (G/L, G/H)$ with the set $$ (G/H)^L =\{ gH \mid g^{-1} L g \leq H \}.$$
The left $W_G H$-action on $\Map_G (G/L, G/H)$ becomes a right action on the set $(G/H)^L$ which is given by $gH \cdot nH= gnH$. It is easy to see that this action is free. Let $\mathcal{G}=\{g_1H, \dots , g_mH\}$ be a set of $W_GH$-orbit representatives of the free $W_GH$-action on $(G/H)^L$. Note that $m$ is equal to the number of $G$-conjugates of $H$ that include $L$.

Since $H$ is cyclic, $L$ is the unique subgroup of $H$ with order equal to $|L|$, so we have $L \leq H \leq N_G (H) \leq N_G(L)$. Also note that if $gH\in (G/H)^L$, then $g \in N_G(L)$. So in our particular situation, we have $(G/H)^L=N_G(L)/H$, and hence $m=|N_G(L):N_G(H)|$.

On the left hand side of the arrow for $\mu$ we have a cartesian product of a free $W_GH$-set with a free $W_GK$-set over $W_GK$. Let $\mathcal X =\{x_1H , \dots , x_s H \}$ be a set of orbit representatives of the free $W_GH$-set $(G/H)^K$. As above we have $(G/H)^K=N_G(K)/H$ and $s=|N_G(K): N_G(H)|$. Similarly, let $\mathcal Y =\{y_1K , \dots , y_t K \}$ be a set of orbit representatives of the free $W_G K$-set $(G/K)^L$. We have $(G/K)^L=N_G(L)/K$ and $t=|N_G(L): N_G(K)|$.

After cancelling the free $W_GK$-orbits, we see that the number of free $W_GH$-orbits on both image and domain of $\mu$ are equal since $st=m$. Hence, to show that $\mu$ induces a bijection, it is enough to show that $\mu$ is surjective. Note that $\mu$ maps the pair $(x_i H, y_j K)$ to $y_j x_i H$.  Let $gH \in (G/H)^L$. Observe that $g K g^{-1}$ is the unique maximal subgroup in $gHg^{-1}$, hence we have $L \leq gK g ^{-1}$. This means $g=y_j n$ for some $y_j K \in \mathcal Y$ and $nK  \in W_GK$. Since $n$ normalizes $K$, we have $K \leq nHn^{-1}$, so $n=x_i n'$ for some $x_i H \in \mathcal X$ and $n' \in W_GH$. This shows that $gH$ is the image of $(x_i n'H, y_j K )$ under $\mu$.
\end{proof}

\begin{definition}\label{def:almost strongly connected G_*}
Let $\cH_*$ be a compatible family of subfamilies. We say $\cH_*$ is \emph{almost strongly connected} if the realization of the poset $\cD _H=\{ d\in \cD \colon H \in \cH_d\}$ is simply connected for all $H \in \cH$ except possibly for some subgroups which are cyclic of prime power order, and for such subgroups $\cD_H$ is either empty or a disjoint union of points.

If $\bV$ factors through a diagram of finite groups $\Gamma _*$ over a one-dimensional diagram $\cD$ and if $\cH_*$ is almost strongly connected, then we say $\bV$ factors through an \emph{almost strongly connected one-dimensional diagram of finite groups $\Gamma _*$}.  
\end{definition}

Now we state our second main result in this section.

\begin{theorem}\label{thm:AlmostConnected} 
Let $G$, $\cH$, and $M$ be as in Theorem \ref{thm:G-actionconst}.
Suppose that we are given a compatible family $\bV$ of unitary representations over $\cH$ which factors through an almost strongly connected one-dimensional diagram of finite groups $\Gamma _*$. Then, the conclusion of Theorem \ref{thm:G-actionconst} still holds.  
\end{theorem}

\begin{proof} We need to show that for every $n\geq 0$, there is $G$-map $E_{\cH} G ^{(n)} \to B_{\cH} (G, \bV ^{\oplus k} )$ for some $k\geq 1$. The rest of the argument follows as in the proof of Theorem \ref{thm:G-actionconst}. 

As in the proof of Theorem \ref{thm:G-bundleconst}, we can consider the homotopy colimit $$Y=\hocolim _{d\in \cD} B_{\cH_d} (G, \bA _d).$$ There is a $G$-map  $\beta: Y \to E_{\cH} G$. Let $Z$ denote the mapping cylinder of $\beta$. 

For every $k\geq 1$, there is a $G$-map $f_k: Y \to  B_{\cH} (G, \bV ^{\oplus k})$. We need to show that for every $n\geq 0$, there is a $k\geq 1$ such that $f_k$ extends to a map $\tilde f_k ^{(n)} : Y \cup Z^{(n)} \to B_{\cH} (G, \bV ^{\oplus k} )$.  The obstructions for extending $\tilde f_k^{(n)}$ to $(n+1)$-skeleton lie in the Bredon cohomology group $$H^{n+1}_G (Z, Y; \pi _n (B_{\cH} (G, \bV ^{\oplus k } )^?))$$ and we need these obstruction groups to be finite for all $n\geq 2$.  

As before we can use the hyper-cohomology spectral sequence to calculate these cohomology groups. The $E_2$-term of this spectral sequence is of the form  $$E_2 ^{p,q} =\Ext _{\bbZ \Gamma _G } ^p (H_q (Z ^{?} , Y^{?}) ; \pi _n (B _{\cH} (G, \bV ^{\oplus k})^?))$$
where $\Gamma _G=\Or _{\cH} G$ is the orbit category over the family $\cH$. So it is enough to show that for every $\bbQ \Gamma_G$-module $M$, the ext-group $$E_2 ^{p,q}=\Ext _{\bbQ \Gamma _G } ^p (H_q (Z ^{?} , Y^{?}; \bbQ) ; M)$$ is zero for all $p,q$ with $p+q \geq 3$.

Let $N_q$ denote $\bbQ \Gamma _G$-module $H_q (Z^{?}, Y^{?}; \bbQ)$. Repeating the argument used in the proof of Lemma \ref{lem:cohcalc}, we see that $$N_q(H)= H_q (Z^H, Y^H; \bbQ)\cong \widetilde H_{q-1} (|\cD _H|;\bbQ)=0$$ for every $H \in \cH$ except possibly when $H$ is a cyclic group of prime power order. When $H$ is a cyclic group of prime power order, $\cD_H$ is either empty or disjoint union of points, so $N_q$ is nonzero only for $q=0,1$. By Lemma \ref{lem:ZeroExt}, $\Ext ^p _{\bbQ \Gamma _G } (N_q , M)=0$ for all $p\geq 2$, so we can conclude that $\Ext ^p _{\bbQ \Gamma _G } (N_q , M)=0$ for all $p,q$ with $p+q \geq 3$. This completes the proof.
\end{proof}

\section{Construction for rank three $p$-groups}\label{sect:RankThreeConst}

In this section we prove Theorem \ref{thm:main}. In the proof we use Theorem \ref{thm:AlmostConnected}, but we first explain how we can reduce the proof of Theorem \ref{thm:main} to the specific situation considered in Theorem \ref{thm:AlmostConnected}.

Let $p$ be an odd prime and $G$ be a rank three $p$-group. In \cite[Theorem 6.7]{unlu-yalcin2}, it is proved that if $G$ acts smoothly on a manifold $M$ with rank one isotropy subgroups, then $G$ acts freely and smoothly on a manifold diffeomorphic to $M \times \bbS ^N$ for some $N>0$. So to prove Theorem \ref{thm:main}, it is enough to prove the following proposition.

\begin{proposition}\label{prop:main} Let $p$ be an odd prime and $G$ be a rank three 
$p$-group.  Then, there exists a smooth $G$-manifold $M$ diffeomorphic to $\bbS^n\times \bbS^m$ for some $n,m>0$, such that for every $x\in M$, the isotropy subgroup $G_x$ has $\rk (G_x)\leq 1$.
\end{proposition} 

To prove Proposition \ref{prop:main}, we use the same strategy as the one used for constructing free rank two $p$-group actions on a product of two spheres. We start with a linear $G$-action on $X=\bbS(V)$ where $V$ is the induced representation $\Ind _{\langle c \rangle} ^G W$, the element $c$ is a central element of order $p$ in $G$, and $W$ is a one-dimensional nontrivial representation of $\langle c \rangle$.

The isotropy subgroups of $G$-action on $X$ satisfy the property that $G_x\cap \la c\ra=1$.  Let $\cH$ denote the set of all subgroups $H\leq G$ such that $H \cap \la c\ra =1$. Note that subgroups in $\cH$ have $\rk (H)\leq 2$. We will prove Proposition \ref{prop:main} by applying Theorem \ref{thm:AlmostConnected} to the manifold $X$ using the family $\cH$. 

There is a further reduction which allows us to focus on rank three $p$-groups with cyclic center. We now explain this reduction. Suppose that the center $Z(G)$ of $G$ has $\rk Z(G)\geq 2$. Then there is a central element $c'\in G$ of order $p$ such that $c'\not \in \langle c\rangle$. Using a one-dimensional nontrivial representation $W' : \la c' \ra \to \bbC ^{\times}$, we can define an induced representation $V'=\Ind _{\la c' \ra} ^G W'$. The $G$-action on $\bbS(V) \times \bbS (V')$ is a smooth action and all its isotropy subgroups have trivial intersections with the central subgroup $\la c, c' \ra \cong \bbZ/p\times \bbZ/p$. This means that all isotropy subgroups of this action have rank $\leq 1$. Hence the conclusion of Proposition \ref{prop:main} holds for the case $\rk Z(G)=2$. Therefore, from now on we can assume that $G$ has cyclic center. 

To prove Proposition \ref{prop:main} we need a compatible family of representations $\bV=\{ V_H \}$ defined on $\cH=\{ H \leq G \colon H \cap Z(G)=1\}$ satisfying the following properties:

\begin{enumerate}
\item $\bV$ factors through an almost strongly connected diagram of finite groups $\Gamma_*$ with associated quadruple $(\Gamma_*, \rho_*, \cH_*, \bA _*)$.   
\item For every rank two elementary abelian subgroup  $E\in \cH$, the $E$-representation $V_E$ is a fixed point free representation.
\end{enumerate}

Note that once we find such a compatible family, the conclusion of Theorem \ref{thm:AlmostConnected} gives a smooth $G$-action on $X \times \bbS ^m$ for some $m\geq 1$, such that isotropy subgroups are the same as the isotropy subgroups of $H$-actions on $\bbS (V_H)$. By the condition (ii) above, this means that all the isotropy subgroups will have rank $\leq 1$. Therefore once we find a compatible family $\bV$ satisfying the properties listed above, the proof of Proposition \ref{prop:main}, and hence the proof of Theorem \ref{thm:main}, will be complete.  
 
As discussed in the introduction, this compatible family comes from an effective class function introduced by Jackson \cite[Proposition 20]{jackson2} in an unpublished work. It was proved later by Klaus \cite[Proposition 3.3]{klaus} in detail that this class function satisfies the desired properties.  Klaus \cite{klaus} used this function to construct a free action on a finite CW-complex homotopy equivalent to a product of three spheres.   

\begin{proposition}\label{pro:ClassFunction}
Let $p$ be an odd prime and $G$ be a rank three $p$-group with cyclic center. Let $\cH$ denote the family of all subgroups $H$ in $G$ such that $H \cap Z(G)=1$. There is a nontrivial class function $\chi : G\to \bbC$ with the following properties: (i) the restriction of $\chi$ to a subgroup $H \in \cH$ is a character of $H$; (ii) for every rank two elementary abelian $p$-subgroup $E\in \cH$ the restriction $\Res ^G_E \chi$ is a character of a fixed point free representation.
\end{proposition} 

\begin{proof} When $p$ is an odd prime, every noncyclic $p$-group has a normal subgroup isomorphic to $C_p\times C_p$ (see \cite[Theorem 4.10]{gorenstein}), hence $G$ has a normal subgroup $Q \cong C_p\times C_p$. Let $C_G(Q)$ denote the centralizer of $Q$ in $G$. Consider the class function  $\chi : G \to \bbC$ defined by 
\begin{equation*} \chi (g)=
\begin{cases} 
p(p-1)|G| & \text{if $g=1$} \\  
0 & \text{if $g\in Z(G)\backslash \{1\}$} \\
-p|G| & \text{if   $g \in Q\backslash Z(Q)$} \\ 
0& \text{if $g\in C_G(Q)\backslash Q$} \\
-|G| & \text{if $g \in G\backslash C_G(Q)$ of order $p$} \\
0 & \text{if $g\in G\backslash C_G(Q)$ of order greater than $p$}.
\end{cases}
\end{equation*}
It can be shown by direct calculation that both statements hold for $\chi$ (see  \cite[Proposition 3.3]{klaus}). 
\end{proof}

Let $\chi$ be the character as in the proof of Proposition \ref{pro:ClassFunction}, and let $\bV_{\chi}$ denote the compatible family of representations defined over $\cH$ such that for every $H\in \cH$, the character for the representation $V_H$ is equal to $\Res ^G _H \chi$. It is clear that the family $\bV_{\chi}$ is a compatible family since it comes from a class function. We claim that $\bV_{\chi}$ satisfies the conditions $(1)$ and $(2)$ listed above, for a suitable choice of quadruple $(\Gamma_*, \rho_*, \cH_*, \bA _*)$. In the rest of this section we introduce the components of this quadruple and show that they satisfy the required properties.

To introduce $\cH_*$, we need to look at the subgroups in $\cH$ more closely. Let $Q$ be a normal subgroup of $G$, isomorphic to $C_p \times C_p$ as in the proof of Proposition \ref{pro:ClassFunction}. Since $Z(G)$ is cyclic, $Z(G)\cap Q=\la c\ra$ is a cyclic group of order $p$. Let $a$ be a non-central element in $Q$. We have $Q=\la c, a\ra \cong \la c\ra \times \la a \ra$.

Let $C_G(Q)$ denote the centralizer of $Q$ in $G$. Since the quotient group $G/C_G(Q)$ acts faithfully on $Q\cong C_p\times C_p$, it must be isomorphic to a subgroup of $GL_2 (\bbF_p)$. Since $|GL_2 (\bbF_p)|=(p^2-1)(p^2-p)$, we can conclude that $|G/C_G (Q)|=p$. Furthermore, we have the following lemma.

\begin{lemma}[See Proposition 3.2 in \cite{klaus}]\label{lem:firsttype}
Let $G$,  $\cH$  and $Q$ be as above. 
If $H\in \cH$ is such that $H \cap Q\neq 1$, then $H \leq C_G(Q)$ and there exists $g\in G$ such that $Q \cap gHg^{-1}=\la a\ra$.
\end{lemma} 

\begin{proof} Since $H\cap \la c\ra=1$, we have $H \cap Q=\la ac^i\ra$ for some $i$. Since $\la ac^i\ra$ is a normal subgroup of order $p$ in $H$, it is a central subgroup of $H$. This means $H$ centralizes $ac^i$, and hence it centralizes $Q$. 
To prove the second statement, let $b\in G$ denote an element such that $b \not \in C_G(Q)$. Then, by replacing $b$ with its power we can assume that
$b^{-1} ab=ac$. This shows that if we take $g=b^i$, then $Q \cap gH g^{-1}=\la a\ra$.  
\end{proof}

We will also need the following lemma.

\begin{lemma}
\label{lem:BEtypes}
Let $H \in \cH$ be such that $H \not \leq C_G(Q)$. Then, $K=H\cap C_G(Q)$ is a cyclic group and $H$  is either cyclic or it is isomorphic to $K \rtimes C_p$ where $C_p$ acts on $K$ either trivially or by the action $k\to k^{1+p^{n-1}}$ where $p^n= |K|$. 
\end{lemma}

\begin{proof} Let $H \in \cH$ be such that $H \not \leq C_G(Q)$. Then, by Lemma \ref{lem:firsttype}, $H \cap Q=1$, in particular, $K \cap Q=1$. This implies that $QK\cong Q \times K$. Since $Q \cong C_p \times C_p$, we must have $\rk (K)\leq 1$, hence $K$ is a cyclic group.  Note that $|H:K|=p$, hence by \cite[Theorem IV.4.1]{brown}, we conclude that $H$ is either cyclic or it is isomorphic to $K \rtimes C_p$ where $C_p$ acts on $K$ either trivially or by the action $k\to k^{1+p^{n-1}}$ where $p^n=|K|$. 
\end{proof}

Now we list all possible types of subgroups in $\cH$ with respect to their relationship to $Q$ and $C_G(Q)$.

\begin{enumerate}[leftmargin=1.8em]
\itemsep5pt \parskip0pt \parsep0pt

\item A subgroup $H \in \cH$ is called a type A subgroup if $H \leq C_G(Q)$. We define the subcollection $\cH _a \subseteq\cH$ as the family of all type A subgroups. Since $C_G(Q)$ is normal in $G$, this is a family, i.e., it is closed under conjugation and taking subgroups.

\item Let $H \in \cH$ be such that $H \not \leq C_G(Q)$. Then, by Lemma \ref{lem:BEtypes},  $H$ is either cyclic or it is isomorphic to $K \rtimes C_p$ where $C_p$ acts on $K$ either trivially or by the action $k\to k^{1+p^{n-1}}$ where $p^n=|K|$. If $K$ is cyclic we call it a type B subgroup, otherwise, we call it a type E subgroup. Note that every type E subgroup has a unique elementary abelian subgroup of rank $2$. This can be easily checked by looking at the subgroup lattice (see also \cite[Lemma 2.1]{mazza}). Let $E_1, \dots, E_m$ denote the conjugacy class representatives of maximal elementary abelian subgroups of type E subgroups. For each $i$, we defined the family $\cH_{e_i}$ as the family of type E subgroups such that $E_i \leq_G H$. Note that if $H \in \cH_{e_i}$, then $H \leq _G N_G(E_i)$.  

\item If $H\in \cH $ such that $H\leq C_G(Q)$ and it is included in a type B or a type E subgroup, then we call it a type C subgroup. Note that type C subgroups are necessarily cyclic.    
\end{enumerate}

\begin{lemma} Let $\cD$ be the discrete poset $\{ a, e_1,\dots, e_m\}$. For each $d\in \cD$, let $\cH_d$ be the subfamily defined as above. Then the diagram of subfamilies $\cH_*$ is almost strongly connected.
\end{lemma}

\begin{proof} Note that the only subgroups $H \in \cH$ which are not in the union $\cH_a \cup (\cup _i \cH _{e_i})$ are type $B$ subgroups, so they are all cyclic. The intersections of families $\cH_d$ for various $d\in \cD$ are easy to describe. We already observed above that if $H \in \cH_a \cap \cH_{e_i}$ for some $i$, then $H$ is a type $C$ subgroup which is again cyclic. Now suppose $H \in \cH_{e_i} \cap \cH_{e_j}$ for some $i\neq j$.
Then $H$ is either cyclic or it is a type $E$ subgroup such that $E_i \leq _G H $ and $E_j \leq _G H$. Since all type $E$ subgroups have a unique elementary abelian rank 2 subgroup, this will imply that $gE_i g^{-1}=E_j$ for some $g\in G$. But the subgroups $E_i$ and $E_j$ were chosen as distinct conjugacy class representatives, so this is not possible. Hence, every subgroup in $\cH_{e_i} \cap \cH_{e_j}$ is cyclic when $i\neq j$. We conclude that $\cH_*$ is almost strongly connected.
\end{proof}

We now describe the diagram of finite groups $\Gamma_*$ and the diagram of families $\bA_*$. For each $d\in \cD$, let $$\Gamma _d=\begin{cases} C_G(Q) \ \text{ if }\ d=a \\N_G(E_i)\ \text{ if }\ d=e_i\ \ . \\ 
\end{cases}$$ Since $\cD$ is a discrete category, it is clear that this is a functor from $\cD$ to finite groups.
For each $d\in \cD$, we define a compatible family of representations $$\bA_d=\{\alpha_H ^d :H \to \Gamma_d \mid H \in \cH_d \}$$ by taking $\alpha ^d _H$ as the 
composition $$\alpha ^d _H: H \maprt{c_g} gHg^{-1} \hookrightarrow \Gamma_d $$
where the conjugation map $c_g$ is defined by  $h \to ghg^{-1}$ and the second map is the inclusion map of $gHg^{-1}$ into $\Gamma _d$. For type E groups, we do this by choosing an arbitrary element $g\in G$ such that $gHg^{-1} \subseteq \Gamma _d$. For type A groups, we take a $g\in G$ such that $Q \cap gHg^{-1}=\la a\ra$. Such an element $g\in G$ always exists by Lemma \ref{lem:firsttype}.

To introduce the collection of representations $\rho_*$, we first introduce some notation. For a $K$-set $X$, where $K \leq G$, we denote by $I_X$ the reduced permutation representation $\bbC X -\bbC$. For example, with this notation, $I_{\la a \ra/1}$ denotes the reduced regular representation of $\la a\ra$.
For each $i=1,\dots, m$, let $C_i$ denote the cyclic subgroup $E_i \cap C_G(Q)$ in $ E_i$. Let  $W_i$ denote the $E_i$-representation $I_{E_i/C_i} + (p-1) I_{E_i/1}$.
For every $d \in \cD$, we define
$$ \rho_d =
\begin{cases} 
n_a\Ind _{\la a\ra}^{C_G(Q)} I_{\la a\ra/1} &  \text{if $d=a$}\\  
n_{e_i}\Ind _{E_i} ^{N_G(E_i)} W_i & \text{if $d=e_i$}. 
\end{cases}$$
The numbers $n_a$ and $n_{e_i}$ are chosen as positive integers such that the equalities 
$$n_a(p-1)\frac{|G|}{p^2}=n_{e_i}(p-1)|N_G(E_j)|=p(p-1)|G|$$
hold. Note that $n_a=p^3$ and $n_{e_i}=p|G|/|N_G(E_i)|$ for all $i$.

\begin{lemma}\label{lem:FactorsThrough} Let  $G$ and $\cH$ be as above, and let $\bV_{\chi}$ denote the compatible family of representations on $\cH$ defined using the class function $\chi$ of Proposition \ref{pro:ClassFunction}. Then $\bV_{\chi}$ factors through a diagram of finite groups $\Gamma _*$ with associated the quadruple  $(\Gamma_*, \rho_*, \cH_*, \bA_* )$ whose components are as introduced above. 
\end{lemma} 

\begin{proof} We only need to show that for every $d\in \cD$, the restriction of $\bV _{\chi }$ to $\cH_d$ is equal to $\rho_d \circ \bA _d$. The rest of the conditions are clear from the construction of the quadruple.

If $d=e_i$ for some $i$, then we need to check that $\Res ^G _{N_G (E_i)} \chi=n_{e_i} \Ind ^{N_G(E_i)} _{E_i} \chi_{W_i}$, where $\chi_{W_i}$ denotes the character for $W_i$. For $g \in N_G(E_i)$,  $$\Bigl (\Ind ^{N_G(E_i)} _{E_i} \chi_{W_i} \Bigr ) (g)=
\begin{cases} 
|N_G(E_i):E_i|\chi_{W_i} (g) & \text{if $g$ has order $p$;}\\
0 & \text{if $g$ has order greater than $p$}.
\end{cases}$$ 
Note that for $g\in E_j$, we have $\chi _{W_i} (g)=0$ if $g\in C_G(Q)$ and $\chi _{W_i} (g)=-p$ if $g\not \in C_G(Q)$. Hence the desired equality holds.

When $d=a$, there is a similar calculation. Observe that if $H\in \cH_a$, then $\alpha _H ^a : H \to \Gamma _a$ is defined by first applying conjugation map $h \to ghg^{-1}$ followed by the inclusion map $gHg^{-1}$ into $\Gamma _a=C_G(Q)$, where the element $g $ is chosen such that $a\in gHg^{-1}$. So it is enough to check whether the equality $\Res ^G _{H} \chi=n_a \Res^{C_G(Q)} _H \Ind ^{C_G(Q)} _{\la a\ra} \chi_a$ holds for a subgroup $H\leq C_G(Q)$ that includes $a$. Here $\chi_a$ denotes the character for $I_{\la a\ra/1}$. If $g \in \la a\ra$, then $$n_a \Bigl ( \Ind ^{C_G(Q)} _{\la a\ra} \chi_a \Bigr )(g)=p^3 (|G|/p^2)\chi _a (g)=\begin{cases} p(p-1)|G| & \text{if $g=1$} \\ -p|G| & \text{if $g\neq 1$}. \end{cases}$$ If $g\in H \backslash \la a \ra$, then the character value is zero. Hence the desired character equality holds.
\end{proof}

\end{document}